\documentclass{amsart}
\usepackage{amsmath}
\usepackage{amssymb}
\usepackage{amsfonts}

\newtheorem{theorem}{Theorem}
\theoremstyle{plain}

\newtheorem{corollary}{Corollary}

\newtheorem{definition}{Definition}

\newtheorem{lemma}{Lemma}

\newtheorem{remark}{Remark}

\numberwithin{equation}{section}
\input{tcilatex}

\begin{document}
\title[The Hermite-Hadamard's inequality for some convex functions ]{The
Hermite-Hadamard's inequality for \ some convex functions via fractional
integrals and related results }
\author{Erhan SET$^{\spadesuit }$}
\address{$^{\spadesuit }$Department of Mathematics, \ Faculty of Science and
Arts, D\"{u}zce University, D\"{u}zce-TURKEY}
\email{erhanset@yahoo.com}
\author{Mehmat Zeki SARIKAYA$^{\spadesuit ,\ast }$}
\email{sarikayamz@gmail.com}
\thanks{$^{\ast }$Corresponding Author}
\author{M.Emin \"{O}zdemir$^{\blacksquare }$}
\address{$^{\blacksquare }$Ataturk University, K. K. Education Faculty,
Department of Mathematics, 25640, Kampus, Erzurum, Turkey}
\email{emos@atauni.edu.tr}
\author{H\"{u}seyin Y\i ld\i r\i m$^{\blacktriangledown }$}
\address{$^{\blacktriangledown }$Department of Mathematics, \ Faculty of
Science and Arts, Kahramanmara\c{s} S\"{u}tc\"{u} Imam University,
Kahramanmara\c{s}-TURKEY}
\email{hyildir@ksu.edu.tr}
\subjclass[2000]{ 26D07, 26D10, 26D15, 26A33}
\keywords{$s-$convex functions, Hadamard's inequality, Riemann-Liouville
fractional integral.}

\begin{abstract}
In this paper, firstly we have established Hermite-Hadamard's inequalities
for $s-$convex functions in the second sense and $m-$convex functions via
fractional integrals. Secondly, a Hadamard type integral inequality for the
fractional integrals are obtained and these result have some relationships
with \cite[Theorem 1, page 28-29]{KBOP}.
\end{abstract}

\maketitle

\section{Introduction}

Let real function $f$ be defined on some nonempty interval $I$ of real line $%
\mathbb{R}$. The function $f$ is said to be convex on $I$ if inequality 
\begin{equation*}
f(\lambda x+(1-\lambda )y)\leq \lambda f(x)+(1-\lambda )f(y)
\end{equation*}%
holds for all $x,y\in I$ and $\lambda \in \left[ 0,1\right] $.

In \cite{hudzik}, Hudzik and Maligranda considered, among others, the class
of functions which are $s-$convex in the second sense. This class of
functions is defined as the following:

\begin{definition}
\label{d1} A function $f:[0,\infty )\mathbb{\rightarrow R}$ is said to be $%
s- $convex in the second sense if 
\begin{equation*}
f(\lambda x+(1-\lambda )y)\leq \lambda ^{s}f(x)+(1-\lambda )^{s}f(y)
\end{equation*}%
for all $x,y\in \lbrack 0,\infty )$, $\lambda \in \lbrack 0,1]$ and for some
fixed $s\in (0,1].$ This class of $s$-convex functions is usually denoted by 
$K_{s}^{2}.$
\end{definition}

It can be easily seen that for $s=1,$ $s-$convexity reduces to ordinary
convexity of functions defined on $[0,\infty ).$

In \cite{GT1}, G. Toader considered the class of $m-$convex functions:
another intermediate between the usual convexity and starshaped convexity.

\begin{definition}
\label{D.1} The function $f:[0,b]\rightarrow \mathbb{R}$, $b>0$, is said to
be $m-$convex, where $m\in \lbrack 0,1],$ if we have%
\begin{equation*}
f(tx+m(1-t)y)\leq tf(x)+m(1-t)f(y)
\end{equation*}%
for all $x,y\in \lbrack 0,b]$ and $t\in \left[ 0,1\right] .$ We say that $f$
is $m-$concave if $(-f)$ is $m-$convex.
\end{definition}

Obviously, for $m=1$ Definition \ref{D.1} recaptures the concept of standard
convex functions on $\left[ a,b\right] ,$ and for $m=0$ the concept
starshaped functions.

One of the most famous inequalities for convex functions is Hadamard's
inequality. This double inequality is stated as follows (see for example 
\cite{Pecaric} and \cite{Dragomir1}): Let $f$ be a convex function on some
nonempty interval $[a,b]$ of real line $\mathbb{R}$, where $a\neq b$. Then 
\begin{equation}
f\left( \frac{a+b}{2}\right) \leq \frac{1}{b-a}\int_{a}^{b}f(x)dx\leq \frac{%
f\left( a\right) +f\left( b\right) }{2}.  \label{E1}
\end{equation}%
Both inequalities hold in the reversed direction if $f$ is concave. We note
that Hadamard's inequality may be regarded as a refinement of the concept of
convexity and it follows easily from Jensen's inequality. Hadamard's
inequality for convex functions has received renewed attention in recent
years and a remarkable variety of refinements and generalizations have been
found (see, for example, \cite{Alomari}-\cite{Set2}).

In \cite{dragomir3}, Hadamard's inequality which for $s-$convex functions in
the second sense is proved by S.S. Dragomir and S. Fitzpatrick.

\begin{theorem}
\label{tt1} Suppose that $f:[0,\infty )\rightarrow \lbrack 0,\infty )$ is an 
$s$-convex function in the second sense, where $s\in (0,1),$ and let $a,b\in
\lbrack 0,\infty ),$ $a<b.$ If $f\in L^{1}(\left[ a,b\right] ),$ then the
following inequalities hold:%
\begin{equation}
2^{s-1}f(\frac{a+b}{2})\leq \frac{1}{b-a}\int\limits_{a}^{b}f(x)dx\leq \frac{%
f(a)+f(b)}{s+1}.  \label{e.1.3}
\end{equation}
\end{theorem}

The constant $k=\frac{1}{s+1}$ is the best possible in the second inequality
in (\ref{e.1.3}).

In \cite{KBOP}, Kirmaci et. al. established a new Hadamard-type inequality
which holds for $s$-convex functions in the second sense. It is given in the
next theorem.

\begin{theorem}
\label{tt2} Let $f:I\rightarrow \mathbb{R}$, $I\subset \lbrack 0,\infty ),$
be a differentiable function on $I^{\circ }$ such that $f^{\prime }\in
L_{1}([a,b]),$ where $a,b\in I,$ $a<b.$ If $\left\vert f^{\prime
}\right\vert ^{q}$ is $s-$convex on $[a,b]$ for some fixed $s\in \left(
0,1\right) $ and $q\geq 1,$ then:%
\begin{equation}
\left\vert \frac{f(a)+f(b)}{2}-\frac{1}{b-a}\int\limits_{a}^{b}f(x)dx\right%
\vert \leq \frac{b-a}{2}\left( \frac{1}{2}\right) ^{\frac{q-1}{q}}\left[ 
\frac{s+\left( \frac{1}{2}\right) ^{s}}{\left( s+1\right) \left( s+2\right) }%
\right] ^{\frac{1}{q}}\left[ \left\vert f^{\prime }(a)\right\vert
^{q}+\left\vert f^{\prime }(b)\right\vert ^{q}\right] ^{\frac{1}{q}}.
\label{e.1.4}
\end{equation}
\end{theorem}

We give some necessary definitions and mathematical preliminaries of
fractional calculus theory \ which are used throughout this paper.

\begin{definition}
Let $f\in L_{1}[a,b].$ The Riemann-Liouville integrals $J_{a+}^{\alpha }f$
and $J_{b-}^{\alpha }f$ of order $\alpha >0$ with $a\geq 0$ are defined by 
\begin{equation*}
J_{a+}^{\alpha }f(x)=\frac{1}{\Gamma (\alpha )}\int_{a}^{x}\left( x-t\right)
^{\alpha -1}f(t)dt,\ \ x>a
\end{equation*}%
and%
\begin{equation*}
J_{b-}^{\alpha }f(x)=\frac{1}{\Gamma (\alpha )}\int_{x}^{b}\left( t-x\right)
^{\alpha -1}f(t)dt,\ \ x<b
\end{equation*}%
respectively where $\Gamma (\alpha )=\int_{0}^{\infty }e^{-t}u^{\alpha -1}du$%
. Here is $J_{a+}^{0}f(x)=J_{b-}^{0}f(x)=f(x).$
\end{definition}

In the case of $\alpha =1,$ the fractional integral reduces to the classical
integral. Properties concerning this operator can be found (\cite{Gorenflo},%
\cite{Miller} and \cite{Podlubni}).

For some recent results connected with fractional integral inequalities see (%
\cite{Anastassiou}-\cite{sarikaya})

In \cite{sarikaya} Sar\i kaya et. al. proved a variant of the identity is
established by Dragomir and Agarwal in \cite[Lemma 2.1]{Dragomir2} for
fractional integrals as the following:

\begin{lemma}
\label{L1} Let $f:\left[ a,b\right] \rightarrow \mathbb{R}$, be a
differentiable mapping on $(a,b)$ with $a<b.$ If $f^{\prime }\in L\left[ a,b%
\right] ,$ then the following equality for fractional integrals holds:%
\begin{eqnarray}
&&\frac{f(a)+f(b)}{2}-\frac{\Gamma (\alpha +1)}{2\left( b-a\right) ^{\alpha }%
}\left[ J_{a+}^{\alpha }f(b)+J_{b-}^{\alpha }f(a)\right]  \notag \\
&&  \label{EE7} \\
&=&\frac{b-a}{2}\int_{0}^{1}\left[ \left( 1-t\right) ^{\alpha }-t^{\alpha }%
\right] f^{\prime }\left( ta+(1-t)b\right) dt.  \notag
\end{eqnarray}
\end{lemma}

The aim of this paper is to establish Hadamard's inequality and Hadamard
type inequalities for $s-$convex functions in the second sense and $m-$conex
functions via Riemann-Liouville fractional integral.

\section{Hermite-Hadamard Type Inequalities for some convex functions via
Fractional Integrals}

\subsection{For $s-$convex functions}

Hadamard's inequality can be represented for $s-$convex functions in
fractional integral forms as follows:

\begin{theorem}
\label{T1} Let $f:\left[ a,b\right] \rightarrow \mathbb{R}$ be a positive
function with $0\leq a<b$ and $f\in L_{1}\left[ a,b\right] .$ If $f$ is a $%
s- $convex mapping in the second sense on $[a,b]$, then the following
inequalities for fractional integrals with $\alpha >0$ and $s\in \left(
0,1\right) $ hold:%
\begin{equation}
2^{s-1}f\left( \frac{a+b}{2}\right) \leq \frac{\Gamma (\alpha +1)}{\left(
b-a\right) ^{\alpha }}\left[ \frac{J_{a+}^{\alpha }f(b)+J_{b-}^{\alpha }f(a)%
}{2}\right] \leq \left[ \frac{1}{\left( \alpha +s\right) }+\beta (\alpha
,s+1)\right] \frac{f(a)+f(b)}{2}  \label{EE3}
\end{equation}%
where $\beta $ is Euler Beta function.
\end{theorem}

\begin{proof}
Since $f$ is a $s-$convex mapping in the second sense on $[a,b]$, we have
for $x,y\in \lbrack a,b]$ with $\lambda =\frac{1}{2}$%
\begin{equation}
f\left( \frac{x+y}{2}\right) \leq \frac{f\left( x\right) +f\left( y\right) }{%
2^{s}}.  \label{E4}
\end{equation}%
Now, let $x=ta+(1-t)b$ and $y=(1-t)a+tb$ with $t\in \left[ 0,1\right] .$
Then, we get by (\ref{E4}) that:%
\begin{equation}
2^{s}f\left( \frac{a+b}{2}\right) \leq f\left( ta+(1-t)b\right) +f\left(
(1-t)a+tb\right)  \label{E5}
\end{equation}%
for all $t\in \left[ 0,1\right] .$

Multiplying both sides of (\ref{E5}) by $t^{\alpha -1}$, then integrating
the resulting inequality with respest to $t$ over $[0,1]$, we obtain%
\begin{eqnarray*}
&&\frac{2^{s}}{\alpha }f\left( \frac{a+b}{2}\right) \\
&& \\
&\leq &\int_{0}^{1}t^{\alpha -1}f\left( ta+(1-t)b\right)
dt+\int_{0}^{1}t^{\alpha -1}f\left( (1-t)a+tb\right) dt \\
&& \\
&=&\frac{1}{\left( b-a\right) ^{\alpha }}\int_{a}^{b}\left( b-u\right)
^{\alpha -1}f(u)du-\frac{1}{\left( a-b\right) ^{\alpha }}\int_{a}^{b}\left(
a-v\right) ^{\alpha -1}f(v)dv \\
&& \\
&=&\frac{\Gamma (\alpha )}{\left( b-a\right) ^{\alpha }}\left[
J_{a+}^{\alpha }f(b)+J_{b-}^{\alpha }f(a)\right]
\end{eqnarray*}%
i.e. 
\begin{equation*}
2^{s-1}f\left( \frac{a+b}{2}\right) \leq \frac{\Gamma (\alpha +1)}{\left(
b-a\right) ^{\alpha }}\left[ \frac{J_{a+}^{\alpha }f(b)+J_{b-}^{\alpha }f(a)%
}{2}\right]
\end{equation*}%
and the first inequality is proved.

For the proof of the second inequality in (\ref{EE3}) we first note that if $%
\ f$ is a $s-$convex mapping in the second sense, then, for $t\in \left[ 0,1%
\right] $, it yields%
\begin{equation*}
f\left( ta+(1-t)b\right) \leq t^{s}f(a)+(1-t)^{s}f(b)
\end{equation*}%
and%
\begin{equation*}
f\left( (1-t)a+tb\right) \leq (1-t)^{s}f(a)+t^{s}f(b).
\end{equation*}%
By adding these inequalities we have%
\begin{equation}
f\left( ta+(1-t)b\right) +f\left( (1-t)a+tb\right) \leq \left[
t^{s}+(1-t)^{s}\right] \left( f(a)+f(b)\right) .  \label{6}
\end{equation}%
Thus, multiplying both sides of (\ref{6}) by $t^{\alpha -1}$and integrating
the resulting inequality with respest to $t$ over $[0,1]$, we obtain%
\begin{eqnarray*}
&&\int_{0}^{1}t^{\alpha -1}f\left( ta+(1-t)b\right) dt+\int_{0}^{1}t^{\alpha
-1}f\left( (1-t)a+tb\right) dt \\
&& \\
&\leq &\left[ f(a)+f(b)\right] \int_{0}^{1}t^{\alpha -1}\left[
t^{s}+(1-t)^{s}\right] dt
\end{eqnarray*}%
i.e. 
\begin{equation*}
\frac{\Gamma (\alpha )}{\left( b-a\right) ^{\alpha }}\left[ J_{a+}^{\alpha
}f(b)+J_{b-}^{\alpha }f(a)\right] \leq \left[ 1+\left( \alpha +s\right)
\beta (\alpha ,s+1)\right] \frac{f(a)+f(b)}{\left( \alpha +s\right) }
\end{equation*}%
where the proof is completed.
\end{proof}

\begin{remark}
\label{R1} If we choose $\alpha =1$ in Theorem \ref{T1}, then the
inequalities (\ref{EE3}) become the inequalities (\ref{e.1.3}) of Theorem %
\ref{tt1}.
\end{remark}

Using Lemma \ref{L1}, we can obtain the following fractional integral
inequality for $s-$convex in the second sense:

\begin{theorem}
\label{T2} Let $f:\left[ a,b\right] \subset \lbrack 0,\infty )\rightarrow 
\mathbb{R}$ be a differentiable mapping on $(a,b)$ with $a<b$ such that $%
f^{\prime }\in L\left[ a,b\right] .$ If $\left\vert f^{\prime }\right\vert
^{q}$ is $s-$convex in the second sense on $[a,b]$ for some fixed $s\in
\left( 0,1\right) $ and $q\geq 1$, then the following inequality for
fractional integrals holds:%
\begin{eqnarray}
&&\left\vert \frac{f(a)+f(b)}{2}-\frac{\Gamma (\alpha +1)}{2\left(
b-a\right) ^{\alpha }}\left[ J_{a+}^{\alpha }f(b)+J_{b-}^{\alpha }f(a)\right]
\right\vert  \notag \\
&&  \label{E10} \\
&\leq &\frac{b-a}{2}\left[ \frac{2}{\alpha +1}\left( 1-\frac{1}{2^{\alpha }}%
\right) \right] ^{\frac{q-1}{q}}  \notag \\
&&  \notag \\
&&\times \left\{ \beta \left( \frac{1}{2};s+1,\alpha +1\right) -\beta \left( 
\frac{1}{2};\alpha +1,s+1\right) +\frac{2^{\alpha +s}-1}{\left( \alpha
+s+1\right) 2^{\alpha +s}}\right\} \left( \left\vert f^{\prime
}(a)\right\vert ^{q}+\left\vert f^{\prime }(b)\right\vert ^{q}\right) ^{%
\frac{1}{q}}.  \notag
\end{eqnarray}
\end{theorem}

\begin{proof}
Suppose that $q=1.$ From Lemma \ref{L1} and using the properties of modulus,
we have 
\begin{eqnarray}
&&\left\vert \frac{f(a)+f(b)}{2}-\frac{\Gamma (\alpha +1)}{2\left(
b-a\right) ^{\alpha }}\left[ J_{a+}^{\alpha }f(b)+J_{b-}^{\alpha }f(a)\right]
\right\vert  \notag \\
&&  \label{EE1} \\
&\leq &\frac{b-a}{2}\int_{0}^{1}\left\vert \left( 1-t\right) ^{\alpha
}-t^{\alpha }\right\vert \left\vert f^{\prime }\left( ta+(1-t)b\right)
\right\vert dt.  \notag
\end{eqnarray}%
Since $\left\vert f^{\prime }\right\vert $ is $s-$convex on $[a,b]$, we have 
\begin{eqnarray}
&&\left\vert \frac{f(a)+f(b)}{2}-\frac{\Gamma (\alpha +1)}{2\left(
b-a\right) ^{\alpha }}\left[ J_{a+}^{\alpha }f(b)+J_{b-}^{\alpha }f(a)\right]
\right\vert  \notag \\
&&  \notag \\
&\leq &\frac{b-a}{2}\int_{0}^{1}\left\vert \left( 1-t\right) ^{\alpha
}-t^{\alpha }\right\vert \left[ t^{s}\left\vert f^{\prime }\left( a\right)
\right\vert +(1-t)^{s}\left\vert f^{\prime }(b)\right\vert \right] dt  \notag
\\
&&  \notag \\
&=&\frac{b-a}{2}\left\{ \int_{0}^{\frac{1}{2}}\left[ \left( 1-t\right)
^{\alpha }-t^{\alpha }\right] \left[ t^{s}\left\vert f^{\prime }\left(
a\right) \right\vert +(1-t)^{s}\left\vert f^{\prime }\left( b\right)
\right\vert \right] dt\right.  \notag \\
&&  \notag \\
&&+\left. \int_{\frac{1}{2}}^{1}\left[ t^{\alpha }-\left( 1-t\right)
^{\alpha }\right] \left[ t^{s}\left\vert f^{\prime }\left( a\right)
\right\vert +(1-t)^{s}\left\vert f^{\prime }\left( b\right) \right\vert %
\right] dt\right\}  \notag \\
&&  \label{E111} \\
&=&\frac{b-a}{2}\left\{ \left\vert f^{\prime }\left( a)\right) \right\vert
\int_{0}^{\frac{1}{2}}t^{s}\left( 1-t\right) ^{\alpha }dt-\left\vert
f^{\prime }\left( a\right) \right\vert \int_{0}^{\frac{1}{2}}t^{s+\alpha
}dt\right.  \notag \\
&&  \notag \\
&&\left. +\left\vert f^{\prime }\left( b)\right) \right\vert \int_{0}^{\frac{%
1}{2}}\left( 1-t\right) ^{s+\alpha }dt-\left\vert f^{\prime }\left(
b)\right) \right\vert \int_{0}^{\frac{1}{2}}\left( 1-t\right) ^{s}t^{\alpha
}dt\right.  \notag \\
&&  \notag \\
&&\left. +\left\vert f^{\prime }\left( a)\right) \right\vert \int_{\frac{1}{2%
}}^{1}t^{\alpha +s}dt-\left\vert f^{\prime }\left( a)\right) \right\vert
\int_{\frac{1}{2}}^{1}t^{s}\left( 1-t\right) ^{\alpha }dt\right.  \notag \\
&&  \notag \\
&&+\left\vert f^{\prime }\left( b)\right) \right\vert \int_{\frac{1}{2}%
}^{1}\left( 1-t\right) ^{s}t^{\alpha }dt-\left\vert f^{\prime }\left(
b)\right) \right\vert \int_{\frac{1}{2}}^{1}\left( 1-t\right) ^{s+\alpha }dt.
\notag
\end{eqnarray}%
Since 
\begin{eqnarray*}
\int_{0}^{\frac{1}{2}}t^{s}\left( 1-t\right) ^{\alpha }dt &=&\int_{\frac{1}{2%
}}^{1}\left( 1-t\right) ^{s}t^{\alpha }dt=\beta \left( \frac{1}{2}%
;s+1,\alpha +1\right) \text{, \ \ \ \ \ \ } \\
&& \\
\int_{0}^{\frac{1}{2}}\left( 1-t\right) ^{s}t^{\alpha }dt &=&\int_{\frac{1}{2%
}}^{1}t^{s}\left( 1-t\right) ^{\alpha }dt=\beta \left( \frac{1}{2};\alpha
+1,s+1\right) \text{,} \\
&& \\
\int_{0}^{\frac{1}{2}}t^{s+\alpha }dt &=&\int_{\frac{1}{2}}^{1}\left(
1-t\right) ^{s+\alpha }dt=\frac{1}{2^{s+\alpha +1}\left( s+\alpha +1\right) }
\end{eqnarray*}%
and%
\begin{equation*}
\int_{0}^{\frac{1}{2}}\left( 1-t\right) ^{s+\alpha }dt=\int_{\frac{1}{2}%
}^{1}t^{s+\alpha }dt=\frac{1}{\left( s+\alpha +1\right) }-\frac{1}{%
2^{s+\alpha +1}\left( s+\alpha +1\right) }.
\end{equation*}%
We obtain 
\begin{eqnarray*}
&&\left\vert \frac{f(a)+f(b)}{2}-\frac{\Gamma (\alpha +1)}{2\left(
b-a\right) ^{\alpha }}\left[ J_{a+}^{\alpha }f(b)+J_{b-}^{\alpha }f(a)\right]
\right\vert \\
&& \\
&\leq &\frac{b-a}{2}\left[ \left\vert f^{\prime }(a)\right\vert +\left\vert
f^{\prime }(b)\right\vert \right] \left\{ \beta \left( \frac{1}{2}%
;s+1,\alpha +1\right) -\beta \left( \frac{1}{2};\alpha +1,s+1\right) +\frac{%
2^{\alpha +s}-1}{\left( \alpha +s+1\right) 2^{\alpha +s}}\right\}
\end{eqnarray*}%
which completes the proof \ for this case. Suppose now that $q>1.$ Since $%
\left\vert f^{\prime }\right\vert ^{q}$ is $s-$convex on $\left[ a,b\right] $%
, we know that for every $t\in \lbrack 0,1]$%
\begin{equation}
\left\vert f^{\prime }(ta+(1-t)b)\right\vert ^{q}\leq t^{s}\left\vert
f^{\prime }(a)\right\vert ^{q}+(1-t)^{s}\left\vert f^{\prime }(b)\right\vert
^{q},  \label{E12}
\end{equation}%
so using well know H\"{o}lder's inequality (see for example \cite{mitrinovic}%
) for $\frac{1}{p}+\frac{1}{q}=1,($ $q>1)$ and (\ref{E12}) in (\ref{EE1}),
we have successively 
\begin{eqnarray*}
&&\left\vert \frac{f(a)+f(b)}{2}-\frac{\Gamma (\alpha +1)}{2\left(
b-a\right) ^{\alpha }}\left[ J_{a+}^{\alpha }f(b)+J_{b-}^{\alpha }f(a)\right]
\right\vert \\
&& \\
&\leq &\frac{b-a}{2}\int_{0}^{1}\left\vert \left( 1-t\right) ^{\alpha
}-t^{\alpha }\right\vert \left\vert f^{\prime }\left( ta+(1-t)b\right)
\right\vert dt \\
&& \\
&=&\frac{b-a}{2}\int_{0}^{1}\left\vert \left( 1-t\right) ^{\alpha
}-t^{\alpha }\right\vert ^{1-\frac{1}{q}}\left\vert \left( 1-t\right)
^{\alpha }-t^{\alpha }\right\vert ^{\frac{1}{q}}\left\vert f^{\prime }\left(
ta+(1-t)b\right) \right\vert dt \\
&& \\
&\leq &\frac{b-a}{2}\left( \int_{0}^{1}\left\vert \left( 1-t\right) ^{\alpha
}-t^{\alpha }\right\vert dt\right) ^{\frac{q-1}{q}}\left(
\int_{0}^{1}\left\vert \left( 1-t\right) ^{\alpha }-t^{\alpha }\right\vert
\left\vert f^{\prime }\left( ta+(1-t)b\right) \right\vert ^{q}dt\right) ^{%
\frac{1}{q}} \\
&& \\
&\leq &\frac{b-a}{2}\left[ \frac{2}{\alpha +1}\left( 1-\frac{1}{2^{\alpha }}%
\right) \right] ^{\frac{q-1}{q}} \\
&& \\
&&\times \left\{ \beta \left( \frac{1}{2};s+1,\alpha +1\right) -\beta \left( 
\frac{1}{2};s+1,\alpha +1\right) +\frac{2^{\alpha +s}-1}{\left( \alpha
+s+1\right) 2^{\alpha +s}}\right\} \left( \left\vert f^{\prime
}(a)\right\vert ^{q}+\left\vert f^{\prime }(b)\right\vert ^{q}\right) ^{%
\frac{1}{q}}
\end{eqnarray*}%
where we use the fact that 
\begin{equation*}
\int_{0}^{1}\left\vert \left( 1-t\right) ^{\alpha }-t^{\alpha }\right\vert
dt=\int_{0}^{\frac{1}{2}}\left[ \left( 1-t\right) ^{\alpha }-t^{\alpha }%
\right] dt+\int_{\frac{1}{2}}^{1}\left[ t^{\alpha }-\left( 1-t\right)
^{\alpha }\right] dt=\frac{2}{\alpha +1}\left( 1-\frac{1}{2^{\alpha }}\right)
\end{equation*}%
which completes the proof.
\end{proof}

\begin{remark}
\label{R3} If we take $\alpha =1$ in Theorem \ref{T2}, then the inequality (%
\ref{E10}) becomes the inequality (\ref{e.1.4}) of Theorem \ref{tt2}.
\end{remark}

\subsection{For $m-$convex functions}

We start with the following theorem:

\begin{theorem}
\label{h1} Let $f:[0,\infty ]\rightarrow 
\mathbb{R}
$ be a positive function with $0\leq a<b$ and $f\in L_{1}[a,b].$ If $f$ is $%
m-convex$ mapping on $[a,b],$ then the following inequalities for fractional
integral with $\alpha >0$ and $m\in (0,1]$ hold:%
\begin{eqnarray}
\frac{2}{\Gamma (\alpha +1)}f\left( \frac{m\left( a+b\right) }{2}\right) 
&\leq &\frac{1}{(mb-ma)^{\alpha }}J_{(ma)^{+}}^{\alpha }f(mb)+\frac{m}{%
(b-a)^{\alpha }}J_{b^{-}}^{\alpha }f(a)  \label{k} \\
&\leq &\frac{f(ma)+m^{2}f(\frac{b}{m})}{(\alpha +1)}+m\frac{f(a)+f(b)}{%
\alpha (\alpha +1)}  \notag
\end{eqnarray}
\end{theorem}

\begin{proof}
Since $f$ is $m-convex$ functions, we have 
\begin{equation*}
f(tx+m(1-t)y)\leq tf(x)+m(1-t)f(y)
\end{equation*}%
and if we choose $t=\frac{1}{2}$ we get%
\begin{equation*}
f\left( \frac{1}{2}\left( x+my\right) \right) \leq \frac{f(x)+mf(y)}{2}.
\end{equation*}%
Now, let $x=mta+m(1-t)b$ and $y=(1-t)a+tb$ with $t\in \lbrack 0,1].$ Then we
get%
\begin{eqnarray}
f\left( \frac{1}{2}\left( mta+m(1-t)b+m(1-t)a+mtb\right) \right) &\leq &%
\frac{f(mta+m(1-t)b)+mf((1-t)a+tb)}{2}  \notag \\
f\left( \frac{1}{2}m\left( a+b\right) \right) &\leq &\frac{%
f(mta+m(1-t)b)+mf((1-t)a+tb)}{2}.  \label{bb}
\end{eqnarray}%
Multiplying both sides of (\ref{bb}) by $t^{\alpha -1},$ then integrating
the resulting inequlity with respect to $t$ over $[0,1],$ we obtain%
\begin{eqnarray*}
f\left( \frac{1}{2}m\left( a+b\right) \right) \overset{1}{\underset{0}{\int }%
}t^{\alpha -1}dt &\leq &\frac{1}{2}\overset{1}{\underset{0}{\int }}t^{\alpha
-1}f(mta+m(1-t)b)dt+\frac{m}{2}\overset{1}{\underset{0}{\int }}t^{\alpha
-1}f((1-t)a+tb)dt \\
\frac{1}{\alpha }f\left( \frac{1}{2}m\left( a+b\right) \right) &\leq &\frac{1%
}{2}\overset{ma}{\underset{mb}{\int }}\left( \frac{u-mb}{ma-mb}\right)
^{\alpha -1}f(u)\frac{du}{m(a-b)}+\frac{m}{2}\overset{b}{\underset{a}{\int }}%
\left( \frac{v-a}{b-a}\right) ^{\alpha -1}f(v)\frac{dv}{b-a} \\
&\leq &\frac{1}{2(mb-ma)^{\alpha }}\overset{mb}{\underset{ma}{\int }}\left(
mb-u\right) ^{\alpha -1}f(u)du+\frac{m}{2}\frac{\Gamma (\alpha )}{%
(b-a)^{\alpha }}J_{b^{-}}^{\alpha }f(a)
\end{eqnarray*}%
which the first inequality is proved.

By the $m$-convexity of $f$, we also have%
\begin{eqnarray*}
&&\frac{1}{2}\left[ f(mta+m(1-t)b)+mf((1-t)a+tb)\right]  \\
&& \\
&\leq &\frac{1}{2}\left[ mtf(a)+m(1-t)f(b)+m(1-t)f(a)+m^{2}f(\frac{b}{m})%
\right] 
\end{eqnarray*}%
for all $t\in \left[ 0,1\right] .$ Multiplying both sides of above
inequality by $t^{\alpha -1}$ and integrating over $t\in \lbrack 0,1],$ we
get%
\begin{eqnarray*}
&&\frac{1}{(mb-ma)^{\alpha }}\overset{mb}{\underset{ma}{\int }}\left(
mb-u\right) ^{\alpha -1}f(u)du+\frac{m}{(b-a)^{\alpha }}\overset{b}{\underset%
{a}{\int }}\left( v-a\right) ^{\alpha -1}f(v)dv \\
&& \\
&\leq &\frac{f(ma)+m^{2}f(\frac{b}{m})}{(\alpha +1)}+m\frac{f(a)+f(b)}{%
\alpha (\alpha +1)}
\end{eqnarray*}%
which this gives the second part of (\ref{k}).
\end{proof}

\begin{corollary}
\label{kk} Under the conditions in Theorem \ref{h1} with $\alpha =1,$ then
the following inequality hold:%
\begin{eqnarray}
f\left( \frac{m\left( a+b\right) }{2}\right)  &\leq &\frac{1}{(b-a)}\overset{%
b}{\underset{a}{\int }}\frac{f(mx)+mf(x)}{2}dx  \label{k1} \\
&\leq &\frac{1}{2}\left[ \frac{f(ma)+m^{2}f(\frac{b}{m})}{2}+m\frac{f(a)+f(b)%
}{2}\right] .  \notag
\end{eqnarray}
\end{corollary}

\begin{remark}
If we take $m=1$ in Corollary \ref{kk}, then the inequalities (\ref{k1})
become the inequalities (\ref{E1}).
\end{remark}

\begin{theorem}
\label{h2} Let $f:[0,\infty ]\rightarrow 
\mathbb{R}
,$ be $m$-convex functions with $m\in (0,1],$ $0\leq a<b$ and $f\in
L_{1}[a,b].$ $F(x,y)_{(t)}:[0,1]\rightarrow 
\mathbb{R}
$ are defined as the following:%
\begin{equation*}
F(x,y)_{(t)}=\frac{1}{2}[f(tx+m(1-t)y)+f((1-t)x+mty)].
\end{equation*}%
Then$,$ we have%
\begin{equation*}
\frac{1}{(b-a)^{\alpha }}\overset{b}{\underset{a}{\int }}(b-u)^{\alpha
-1}F\left( u,\frac{a+b}{2}\right) _{\left( \frac{b-u}{b-a}\right) }du\leq 
\frac{\Gamma (\alpha )}{2(b-a)^{\alpha }}J_{a^{+}}^{\alpha }f(b)+\frac{m}{%
2\alpha }f\left( \frac{a+b}{2}\right) 
\end{equation*}%
for all $t\in \lbrack 0,1].$
\end{theorem}

\begin{proof}
Since $f$ and $g$ are $m-convex$ functions, we have%
\begin{eqnarray*}
F(x,y)_{(t)} &\leq &\frac{1}{2}\left[ tf(x)+m(1-t)f(y)+(1-t)f(x)+mtf(y)%
\right]  \\
&=&\frac{1}{2}\left[ f(x)+mf(y)\right] 
\end{eqnarray*}%
and so,%
\begin{equation*}
F\left( x,\frac{a+b}{2}\right) _{(t)}\leq \frac{1}{2}\left[ f(x)+mf\left( 
\frac{a+b}{2}\right) \right] .
\end{equation*}%
If we choose $x=ta+(1-t)b,$ we have%
\begin{equation}
F\left( ta+(1-t)b,\frac{a+b}{2}\right) _{(t)}\leq \frac{1}{2}\left[
f(ta+(1-t)b)+mf\left( \frac{a+b}{2}\right) \right] .  \label{AA}
\end{equation}%
Thus multiplying both sides of (\ref{AA}) by $t^{\alpha -1}$, then
integrating the resulting inequality with respect to $t$ over $[0,1],$ we
obtain%
\begin{equation*}
\underset{0}{\overset{1}{\int }}t^{\alpha -1}F\left( ta+(1-t)b,\frac{a+b}{2}%
\right) _{(t)}dt\leq \frac{1}{2}\left[ \underset{0}{\overset{1}{\int }}%
t^{\alpha -1}f(ta+(1-t)b)dt+\underset{0}{\overset{1}{\int }}t^{\alpha
-1}mf\left( \frac{a+b}{2}\right) dt\right] .
\end{equation*}%
Thus, if we use the change of the variable $u=ta+(1-t)b,\ t\in \left[ 0,1%
\right] ,$ then have the conclusion.
\end{proof}


\begin{thebibliography}{99}
\bibitem{Alomari} M. Alomari and M. Darus, \textit{On the Hadamard's
inequality for log-convex functions on the coordinates},\ Journal of
Inequalities and Applications, vol. 2009, Article ID 283147, 13 pages, 2009.

\bibitem{AAG} A.G. Azpeitia, \textit{Convex functions and the Hadamard
inequality}, Rev. Colombiana Math., 28 (1994), 7-12.

\bibitem{BOP} M.K. Bakula, M.E. \"{O}zdemir, J. Pe\v{c}ari\'{c}, \textit{%
Hadamard tpye inequalities for }$m-$\textit{convex \ and }$\left( \alpha
,m\right) $\textit{-convex functions}, J. Ineq. Pure and Appl. Math., 9(4)
(2008), Art. 96.

\bibitem{Bakula} M. K. Bakula and J. Pe\v{c}ari\'{c}, \textit{Note on some
Hadamard-type inequalities}, Journal of Inequalities in Pure and Applied
Mathematics, vol. 5, no. 3, article 74, 2004.

\bibitem{Dragomir1} S. S. Dragomir and C. E. M. Pearce, \textit{Selected
Topics on Hermite-Hadamard Inequalities and Applications}, RGMIA Monographs,
Victoria University, 2000.

\bibitem{Dragomir2} S. S. Dragomir and R.P. Agarwal, \textit{Two
inequalities for differentiable mappings and applications to special means
of real numbers and to trapezoidal formula}, Appl. Math. lett., 11(5)
(1998), 91-95.

\bibitem{Dragomir} S.S. Dragomir, \textit{On some new inequalities of
Hermite-Hadamard type for }$m-$\textit{convex functions,} Tamkang J. Math.,
3(1) (2002).

\bibitem{dragomir3} S. S. Dragomir and S. Fitzpatrik, \textit{The Hadamard's
inequality for\ }$s$\textit{-convex functions in the second sense,
Demonstratio Math.} 32(4), (1999), 687-696.

\bibitem{Gill} P. M. Gill, C. E. M. Pearce, and J. Pe\v{c}ari\'{c}, \textit{%
Hadamard's inequality for }$\mathit{r}$\textit{-convex functions},\ Journal
of Mathematical Analysis and Applications, vol. 215, no. 2, pp. 461--470,
1997.

\bibitem{hudzik} H. Hudzik and L. Maligranda, \textit{Some remarks on }$s-$%
\textit{convex functions}, Aequationes Math. 48 (1994), 100-111.

\bibitem{KBOP} U.S. Kirmaci, M.K. Bakula, M.E. \"{O}zdemir, J. Pe\v{c}ari%
\'{c}, \textit{Hadamard-tpye inequalities for }$s$\textit{-convex functions}%
, Appl. Math. and Comp., 193 (2007), 26-35.

\bibitem{GT1} G.Toader, Some generalizations of the convexity,\textit{\
Proceedings of The Colloquium On Approximation} \textit{And Optimization,}
Univ. Cluj-Napoca, Cluj-Napoca,1985, 329-338.

\bibitem{Ozdemir1} M. E. \"{O}zdemir, M. Avci, and E. Set, \textit{On some
inequalities of Hermite-Hadamard type via }$\mathit{m}$\textit{-convexity},\
Applied Mathematics Letters, vol. 23, no. 9, pp. 1065--1070, 2010.

\bibitem{Pecaric} J.E. Pe\v{c}ari\'{c}, F. Proschan and Y.L. Tong, \textit{%
Convex Functions, Partial Orderings and Statistical Applications}, Academic
Press, Boston, 1992.

\bibitem{Set1} E. Set, M. E. \"{O}zdemir, and S. S. Dragomir, \textit{On the
Hermite-Hadamard inequality and other integral inequalities involving two
functions},\ Journal of Inequalities and Applications, Article ID 148102, 9
pages, 2010.

\bibitem{Set2} E. Set, M. E. \"{O}zdemir, and S. S. Dragomir,\textit{\ On
Hadamard-Type inequalities involving several kinds of convexity},\ Journal
of Inequalities and Applications, Article ID 286845, 12 pages, 2010.

\bibitem{Anastassiou} G. Anastassiou, M.R. Hooshmandasl, A. Ghasemi and F.
Moftakharzadeh, \textit{Montogomery identities for fractional integrals and
related fractional inequalities}, J. Ineq. Pure and Appl. Math., 10(4)
(2009), Art. 97.

\bibitem{Belarbi} S. Belarbi and Z. Dahmani, \textit{On some new fractional
integral inequalities}, J. Ineq. Pure and Appl. Math., 10(3) (2009), Art. 86.

\bibitem{Dahmani1} Z. Dahmani, \textit{New inequalities in fractional
integrals}, International Journal of Nonlinear Scinece, 9(4) (2010), 493-497.

\bibitem{Dahmani2} Z. Dahmani, \textit{On Minkowski and Hermite-Hadamard
integral inequalities via fractional integration}, Ann. Funct. Anal. 1(1)
(2010), 51-58.

\bibitem{Dahmani3} Z. Dahmani, L. Tabharit, S. Taf, \textit{Some fractional
integral inequalities}, Nonl. Sci. Lett. A, 1(2) (2010), 155-160.

\bibitem{Dahmani4} Z. Dahmani, L. Tabharit, S. Taf, \textit{New
generalizations of Gruss inequality usin Riemann-Liouville fractional
integrals}, Bull. Math. Anal. Appl., 2(3) (2010), 93-99.

\bibitem{Gorenflo} R. Gorenflo, F. Mainardi, \textit{Fractional calculus:
integral and differential equations of fractional order}, Springer Verlag,
Wien (1997), 223-276.

\bibitem{Miller} S. Miller and B. Ross, \textit{An introduction to the
Fractional Calculus and Fractional Differential Equations}, John Wiley \&
Sons, USA, 1993, p.2.

\bibitem{Podlubni} I. Podlubni, \textit{Fractional Differential Equations},
Academic Press, San Diego, 1999.

\bibitem{sarikaya1} M.Z. Sarikaya and H. Ogunmez, \textit{On new
inequalities via Riemann-Liouville fractional integration},
arXiv:1005.1167v1, submitted.

\bibitem{sarikaya} M.Z. Sar\i kaya, E. Set, H. Yald\i z and N. Ba\c{s}ak, 
\textit{Hermite-Hadamard's inequalities for fractional integrals and related
fractional inequalities}, Mathematical and Computer Modelling, accepted.
\end{thebibliography}
\end{document}